\documentclass[10pt,leqno]{amsart}
\usepackage{amsfonts}
\usepackage{amsmath}
\usepackage{amssymb}
\usepackage{amsthm}
\usepackage{amsbsy}
\usepackage{amsrefs}
\usepackage{xcolor}
\usepackage{graphicx} \usepackage{enumerate} \usepackage{multicol}
\usepackage{mathrsfs} \usepackage[all,cmtip]{xy}
\usepackage{soul}
\usepackage{lscape}

\newcounter{dummy} \numberwithin{dummy}{section}
\newtheorem{theorem}[dummy]{Theorem}
\newtheorem{corollary}[dummy]{Corollary}
\newtheorem{lemma}[dummy]{Lemma}

\newtheorem{proposition}[dummy]{Proposition}
\theoremstyle{remark}
\newtheorem{remark}[dummy]{Remark}
\newtheorem{example}[dummy]{Example}

\DeclareMathOperator{\Ann}{Ann}
\DeclareMathOperator{\pr}{pr}
\DeclareMathOperator{\tr}{tr}
\DeclareMathOperator{\dv}{div}
\DeclareMathOperator{\sign}{sign}
\DeclareMathOperator{\spn}{span}
\DeclareMathOperator{\II}{II}
\DeclareMathOperator{\Sec}{Sec}
\DeclareMathOperator{\cha}{char}
\newcommand{\ve}{\varepsilon}


\renewcommand{\theequation}{\arabic{section}.\arabic{equation}}
\numberwithin{equation}{section}

\title[A sub-Riemannian Gauss-Bonnet theorem]{A sub-Riemannian Gauss-Bonnet theorem for surfaces in contact manifolds}
\author[E.~Grong, J.~Hidalgo and S.~Vega-Molino]{Erlend Grong, Jorge Hidalgo and Sylvie Vega-Molino}

\thanks{The first and third authors are supported by the grant GeoProCo from the Trond Mohn Foundation - Grant TMS2021STG02 (GeoProCo). The second author is supported by the State Research Agency (AEI) via the grant no. PID2020-117868GB-I00, funded by MCIN/AEI/10.13039/ 501100011033/, Spain.}

\address{University of Bergen, Dep. of Mathematics, P.O.~Box 7803, 5020 Bergen, Norway}
\email{erlend.grong@uib.no}
\address{Departamento de Geometría y Topología, Universidad de Granada, 18071, Granada, Spain}
\email{jorgehcal@ugr.es}  
\address{University of Bergen, Dep. of Mathematics, P.O.~Box 7803, 5020 Bergen, Norway}
\email{sylvie.vega-molino@uib.no}

\begin{document}

\begin{abstract}
We obtain a sub-Riemannian version of the classical Gauss-Bonnet theorem. We consider subsurfaces of a three dimensional contact sub-Riemannian manifolds, and using a family of taming Riemannian metric, we obtain a pure sub-Riemannian result in the limit.  In particular, we are able to recover topological information of the surface from the geometry around the characteristic set, i.e., the points where the tangent space to the surface and contact structure coincide. We both give a version for surfaces without boundary and surfaces with boundary.
\end{abstract}

\maketitle

\section{Introduction}
The classical Gauss-Bonnet theorem shows that it is possible to recover purely topological information of a surface from the choice of a smooth structure and Riemannian metric. In this paper, we want to show that such topological information can also be obtained from the restriction of a sub-Riemannian geometric structure, even though the induced metric on the surface will not induce the manifold topology. Consider a three dimensional, connected manifold $M$ with a contact distribution $E$. A smoothly varying inner product $g_E$ defined only on $E$ is called a sub-Riemannian metric. Such a geometric structure induces a distance $d_{g_E}$ on $M$, which, although not Lipschitz equivalent to any Riemannian distance, will induce the manifold topology on $M$ \cite{Bel96,Mon02}. Given an orientation of $E$, there is a unique choice of Reeb vector field $Z$ defined by $E$ and $g_E$. We can then extend the sub-Riemannian metric to a Riemannian metric $g_\ve = \langle \cdot , \cdot \rangle_\ve$ on $M$ with $\|Z\|_\ve = 1/\sqrt{\ve}$ such that the length of all vectors outside of $E$ go to infinity as $\ve \downarrow 0$. Then its Riemannian distance $d_{g_\ve}$ converge to $d_{g_E}$, and this convergence is uniform on compact sets \cite{BGRV19}.

If we restrict ourselves to a subsurface $\Sigma \subseteq M$, then the picture is quite different. If $h_\ve$ is the induced metric on $\Sigma$ from $g_\ve$, then $d_{h_\ve}$ does not converge to a metric compatible with the topology if it converges at all, see \cite{BBC20,BBCh21} for details. Seeing that this limit breaks the topology of $\Sigma$, it is interesting to study the limit of the Gauss-Bonnet formula with respect to $h_\ve$ as $\ve \downarrow 0$. See \cite{BG19,BGVM22} for previous results relating topology to sub-Riemannian invariants. One of our main inspirations \cite{BTV17,BTV20}, in which the authors consider a surface embedded in the Heisenberg group and determine a partial Gauss-Bonnet result. 

We can state our main result for compact surfaces without boundary as follows. Surfaces with boundary are considered in Section~\ref{sec:Boundary}. Assume that both $M$ and~$E$ are orientable, and choose an orientation of $E$. There then exists a unique positively oriented contact form $E = \ker \alpha$, such that $\| d\alpha\|=1$. Let $\Sigma\subseteq M$ be a compact oriented $C^2$-subsurface, with the characteristic set
\begin{align*}
\cha(\Sigma) & = \{ x \in \Sigma \, : \, T_x \Sigma \subseteq E_x \}. 
\end{align*}
This set will be contained in a one-dimensional $C^1$-submanifold of $\Sigma$ as shown in \cite[Lemma~2.4]{BBC20}. Observe that this set does not depend on the metric $g_E$. Let $h$ denote the restriction to $\Sigma$ of $g = g_1$ with area form $\sigma = \sigma^1$, and let $\beta$ be the restriction of the contact form $\alpha$. We then define a function $a: \Sigma \to \mathbb{R}$ by
$$d\beta = - a \sigma,$$
which gives the oriented angle between $\wedge^2 E_x$ and $\wedge^2 T_x \Sigma$. Note that the points $x \in \Sigma$ with $|a(x)| =1$ correspond to the set $\cha(\Sigma)$.  For our result, we need the following assumption to hold. Define $\Phi(c) = \int_{0\leq a \leq c} \sigma$ and $\Phi(-c) = \int_{c \leq a \leq 0} \sigma$ for $0 \leq c \leq 1$.
\begin{equation} \tag{A} \label{A} 
\begin{array}{c}
\text{The function $\Phi(c)$ is differentiable with bounded derivative} \\ \text{ for some interval $(-1,-c_1) \cup (c_1,1)$.} \end{array}
\end{equation}
We will give a sufficient condition for \eqref{A} to hold in Section~\ref{sec:Asufficient}.

On the set $0 < |a(x)| < 1$, it is possible to define an $h$-unit vector field $X$ on $\Sigma$ with values in $T\Sigma \cap E$, giving us a corresponding positively oriented orthonormal basis $X$, $X_2$. Define a measure of the curvature of $\Sigma$ using $a$ and the Reeb vector field $Z$,
\begin{align*}
K_{\Sigma,E} & = -  \left(\dv(X_2)  - \sqrt{1-a^2} \langle [Z,X], X \rangle   \right)X_2 a \\
& \qquad - \left(  1+ \frac{\sqrt{1-a^2}}{a} \dv(X) \right)^2 + 1-a^2 + \frac{\sqrt{1-a^2}}{a}\dv(X),
\end{align*}
where the divergence is with respect to the volume form $\sigma$. The function $K_{\Sigma, E}$ does not depend on the orientation of $E$, and we will show that it is uniformly bounded close to the characteristic set $\cha(\Sigma)$. We will use this function to state our main result.
\begin{theorem}[Sub-Riemannian Gauss-Bonnet theorem] \label{th:main}
Let $\Sigma\subseteq M$ be a compact $C^2$-surface without boundary such that Assumption \eqref{A} holds. Then
$$\left. \frac{d}{dc} \right|_{c=0} \int_{|a| > 1 - c} K_{\Sigma,E} (x) \sigma(x)= 2\pi \chi(\Sigma).$$
\end{theorem}
In particular, $\chi(\Sigma)$ can be determined by the values of $K_{\Sigma,E}$ in a neighborhood of $\cha(\Sigma)$. We get the following simple corollary
\begin{corollary}
\begin{enumerate}[\rm (a)]
\item If $\cha(\Sigma) = \emptyset$, then $\chi(\Sigma) =0$.
\item If $K_{\Sigma,E}$ is non-negative in a neighborhood of $\cha(\Sigma)$, then $\Sigma$ is homeomorphic to a sphere or a torus.
\end{enumerate}
\end{corollary}

The structure of the paper is as follows. In Section~\ref{sec:3dContact}, we give the basic definitions related to sub-Riemannian 3-dimensional contact manifolds $(M,E,g)$ with a surface $\Sigma \subseteq M$. We will also give a sufficient condition for \eqref{A} in Section~\ref{sec:Asufficient}. In Section~\ref{sec:Variational}, we show how curvature of the taming Riemannian metric $g_\ve$ varies with $\ve$. We will continue with computations by looking at the metric $h_\ve$ on $\Sigma$ from the restriction of $g_\ve$ in Section~\ref{sec:Embedded}. In Section~\ref{sec:ProofGB}, we give the proof of Theorem~\ref{th:main} in several steps. With some minor restrictions on the boundary, we give a Gauss-Bonnet theorem for surfaces with a piecewise $C^2$-boundary in Section~\ref{sec:Boundary}.

\subsection{Relation to previous works} \label{sec:previous}
Let $K^\ve$ denote the Gaussian curvature denote the Gaussian curvature of $\Sigma$ with respect to $h_\ve$, and let $\sigma^\ve$ be its volume form. If we consider the limit of the equation $\int_{\Sigma} K^\ve d\sigma^\ve = 2\pi \chi(\Sigma)$ as $\ve \to 0$, then away from the characteristic set, $K^\ve d\sigma^\ve$ only has terms of half-integer order with respect to $\ve$, starting with order $-1/2$, see Section~\ref{sec:ProofGB} for details. Isolating this part of order $-1/2$, and using that $\sqrt{\ve} \int_{\Sigma} K^\ve d\sigma^\ve =0$, we obtain an identity for a part of the curvature that has vanishing average over $\Sigma$. This observation is found in our paper in Remark~\ref{re:vanish}. Such results have appeared previously for given ambient contact manifolds such as Heisenberg group and $\mathrm{SE}(2)$, see \cite{DV16,BTV17,Liu_Miao_Li_Guan_2021,Guan_Liu_2021,Li_Liu_2022,Veloso_2023a} for examples.
These examples include assumptions to avoid any contribution from the characteristic set $\cha(\Sigma)$. For example, in \cite{BTV17}, the surface $\Sigma$ is defined locally as the level set of a function $u$, and it is assumed that $\frac{1}{\| \nabla^E u\|}$ is integrable near the characteristic set, which in our notation is equivalent to assuming that $\frac{1}{\sqrt{1-a^2}}$ is integrable close to the set $|a| =1$. See Section~\ref{sec:Local} for details.

Our paper differ in that we are mainly interested in how the topological Euler characteristic $\chi(\Sigma)$ is preserved under the limit. We are hence interested in the term of order zero of the integral $\int_{\Sigma} K^\ve d\sigma^\ve$, which can only be obtained through the study of the integral close to $\cha(\Sigma)$. We also have no restriction on the ambient sub-Riemannian contact manifold $(M, E,g)$, apart from assumption \eqref{A}.

Finally, we also mention similar work in \cite{Agrachev_Boscain_Sigalotti_2008}, on almost-Riemannian manifolds such as the Gruhsin plane, which can be considered as the case where $h^\ve$ approaches a Riemannian metric outside a small singular set.

\section{Contact manifolds and the horizontal angle parameter} \label{sec:3dContact}
\subsection{Contact distributions and the Reeb vector field} \label{sec:ContactReeb}
Let $M$ be a three dimensional manifold, with $E$ being a rank two, contact distribution. In other words, we have $E + [E,E] = TM$. For simplicity, we will assume that both~$M$ and~$E$ are orientable and with chosen orientations. It follows that the subbundle $\Ann E \subseteq T^*M$ of covectors vanishing on $E$ is orientable as well. Let $E$ be equipped with a fiber metric $g_E$ making $(M, E, g_E)$ into a sub-Riemannian manifold. Let $\alpha$ be the unique non-vanishing section of $\Ann(E)$ satisfying $d\alpha(u,v) = -1$ for any positively oriented orthonormal basis $u,v \in E_x$ and any~$x \in M$. \emph{The Reeb vector field} is then the unique vector field $Z$ satisfying
$$\alpha(Z) = 1, \qquad d\alpha(Z, \, \cdot \, ) = 0.$$
We will use the Reeb vector field to extend $g_E$ to a Riemannian metric. Consider a Riemannian metric $g$ on $M$ such that $g |_E = g_E$ and such that $Z$ is orthogonal to $E$ and a unit vector field. We write $g = \langle \, \cdot \, , \, \cdot \, \rangle$.

Introduce a tensor $J: TM \to E$ by $u,v \in T_x M$, $x\in M$,
$$d\alpha(v,w)  = \langle v, Jw \rangle = - \langle Jv, w \rangle.$$
By our definition of $\alpha$ it follows that $J$ is an almost complex structure when restricted to $E$, corresponding to a rotation of $\frac{\pi}{2}$ in the positive direction.
We finally introduce a symmetric tensor $\tau: TM \to E\subseteq TM$ by
$$\langle \tau v, w \rangle = \langle v, \tau w \rangle = \frac{1}{2} (\mathcal{L}_Z g)(\pr_E v, \pr_E w).$$
We emphasize that from the previous definitions $\tau Z = JZ = 0$.

We want a connection such that both $E$ and $E^\perp$ are parallel and with the torsion as simple as possible. We define a connection $\nabla$ such that if $Z$ is the Reeb vector field and $Y_1$ and $Y_2$ are arbitrary sections of $E$, then
$$\nabla Z = 0, \qquad \nabla_Z Y_1 = [Z,Y_1] + \tau Y_1, \qquad \nabla_{Y_1} Y_2  = \pr_E \nabla^g_{Y_1} Y_2,$$
where $\pr_E$ is the $g$-projection to $E$ and $\nabla^g$ is the Levi-Civita connection of $g$. This connection is compatible with $g$ and has torsion
\begin{align*}
T(V, W)  & = -  \langle JV, W \rangle Z + \alpha(V) \tau W - \alpha(W) \tau V, & & V, W \in \Gamma(TM).
\end{align*}
See \cite{Tanno89,Hladky12} for more about the choice of connection in this setting, and this connection in particular. We have the following identity.

\begin{remark} \label{re:identities}
We note the following identities
\begin{equation} \label{identities} \nabla J = 0, \qquad \tr_g \tau = 0, \qquad \tau J = - J \tau.\end{equation}
We also observe that since $\mathcal{L}_Z d\alpha = d \mathcal{L}_Z \alpha =0$, we have $\mathcal{L}_Z J = 2 \tau J$. For proof, see e.g., \cite{Gro20}.
\end{remark}

\subsection{Surfaces and the horizontal angle parameter} \label{sec:Xvector}
Let $\Sigma$ be an oriented, compact surface embedded into $(M,E,g_E)$ without boundary. Recall that $\alpha$ is the contact form and $g$ is the extension to a Riemannian metric using the Reeb vector field. We write $h = g |_{T\Sigma}$ for the induced Riemannian metric on $\Sigma$ from the metric~$g$, with $\sigma$ being the corresponding volume form.  We will use the following notation for the rest of the paper. We define \emph{the horizontal angle parameter} $a(x)$ at $x \in \Sigma$ as the inner product between the positively oriented normal vectors of $T_x\Sigma$ and $E_x$. If $v_1, v_2$ and $w_1, w_2$ are positively oriented orthonormal bases of respectively $T_x \Sigma$ and $E_x$, then
$$a(x) = \langle v_1 \wedge v_2, w_1 \wedge w_2 \rangle = - d\alpha(w_1, w_2).$$
We observe that $|a(x)| =1$ exactly at $\cha(\Sigma)$, the set of points $x \in \Sigma$ where $T_x \Sigma = E_x$. 
Write
\begin{equation} \label{Sigmac}
\Sigma_{c_1 \leq c_2} = \{x \in \Sigma \, : \,  c_1 \leq a(x) \leq c_2 \}, \quad \Sigma_{c} = \Sigma_{c \leq c} \quad \Sigma' = \Sigma \setminus \cha(\Sigma).
\end{equation}
\begin{lemma}
Let $\beta = \alpha|_{T\Sigma}$ be the restriction of the contact form to $\Sigma$.
\begin{enumerate}[\rm (a)]
\item We have identity $d\beta = - a \sigma$ and $\|\beta\| =\sqrt{1-a^2}$.
\item There exists a unique unit vector field $X$ on $\Sigma'$ with values in $T\Sigma' \cap E|_{\Sigma'}$ such that
$$\sqrt{1-a^2} \iota_X \sigma = \beta.$$
Furthermore, the unique vector field $X_2$ such that $X, X_2$ is a positive orthonormal basis is given by
$$X_2 = \sqrt{1-a^2} Z + a JX.$$
\end{enumerate}
\end{lemma}
\begin{proof}
\begin{enumerate}[\rm (a)]
\item This follows directly from the definition of $a$.
\item Since $\beta$ is non-vanishing on $\Sigma'$, we obtain that $\ker \beta$ is a one-dimensional line-bundle spanned by some unit vector field $X$, uniquely determined up to sign. As $\iota_X \sigma$ is a non-vanishing function that vanishes on $\spn\{ X\}$, it follows that $\iota_X \sigma = \varphi_1 \beta$ for some non-vanishing function $\varphi_1$. By requiring $\varphi_1 > 0$, the choice of $X$ is unique. Furthermore, if $X$, $X_2$ is a positively oriented orthonormal basis, then
$$\sigma(X, X_2) = 1 = (\iota_X \sigma)(X_2) = \varphi_1 \beta(X_2), \qquad d\beta(X, X_2) = -a.$$
It follows that $X_2 = \varphi_1 Z + a JX$ and since $X_2$ must have length 1, we finally have $\varphi_1 = \sqrt{1-a^2}$.
\end{enumerate}
\end{proof}

To simplify notations later, we will also introduce the following functions. Define $\tau_0, \tau_1: \Sigma' \to \mathbb{R}$ by $\tau_0 = \langle \tau X, X \rangle = - \langle \tau JX, JX \rangle$ and $\tau_1 = \langle \tau X, JX \rangle$. We remark the following relationship.
\begin{lemma}
\begin{equation}
\label{XaIdentity} \frac{Xa}{\sqrt{1 -a^2}}  = a^2  + \sqrt{1-a^2}\langle \nabla_{X_2} X , JX \rangle -(1-a^2) \tau_1.
\end{equation}
\end{lemma}
\begin{proof}
We consider the normal vector $N = a Z - \sqrt{1-a^2} JX$. Since $T\Sigma$ is Frobenius integrable, we have $\langle N, [X, X_2] \rangle =0$. The bracket is
\begin{align*}
[X,X_2] & \textstyle =- \frac{a Xa}{\sqrt{1-a^2}} Z+  Xa JX + \sqrt{1-a^2} [X,Z] + a [X,JX]
\end{align*}
and so
\begin{align*}
0 = \langle N, [X,X_2] \rangle & \textstyle =-  \frac{a^2 Xa}{\sqrt{1-a^2}} + a^2 - \sqrt{1-a^2} Xa \\
& \textstyle \qquad  + \sqrt{1-a^2} \langle \sqrt{1-a^2} [Z,X] + a [JX,X], JX \rangle \\
&\textstyle =-  \frac{ Xa}{\sqrt{1-a^2}} + a^2   + \sqrt{1-a^2} \langle \nabla_{X_2} X, JX \rangle - (1-a^2) \langle \tau X, JX \rangle .
\end{align*}
The result follows.
\end{proof}

\begin{remark}[Characteristic vector fields] \label{re:Characteristic}
Let us compare this basis $X$, $X_2$ to previous work in \cite[Chapter~4.6]{Geiges08}, see also \cite{BBC22}.
\emph{A characteristric vector field} on~$\Sigma$ is a vector field $\tilde X$ with values in $E \cap T \Sigma$ that vanishes on~$\cha(\Sigma)$, while also satisfying
$$\dv(\tilde X)(x) \neq 0, \qquad x \in \cha(\Sigma).$$
We get such a globally defined vector field by $\tilde X = \sqrt{1-a^2} X$, which will then satisfy $\iota_{\tilde X} \sigma = \beta$ and
$$\dv(\tilde X) \sigma = d\iota_{\tilde X} \sigma= d\beta = - a\sigma.$$
Hence the horizontal angle parameter can be seen as the negative divergence of this characteristic vector field $\tilde X$. Observe that for the vector field $X$, we will hence have
$$\dv(X) = \frac{a}{\sqrt{1-a^2}} \left( \frac{Xa}{1-a^2} - 1 \right).$$
\end{remark}

%\begin{align*}
%T(X_2, X) & = -a Z+ \sqrt{1-a^2} \tau X \\
%\nabla_X X_2 & = - \frac{Xa}{\sqrt{1-a^2}} Z  + Xa JX + a J\nabla_X X\\ 
%[X,X_2] & \textstyle =- \frac{a Xa}{\sqrt{1-a^2}} Z+  Xa JX + \sqrt{1-a^2} [X,Z] + a [X,JX] \\
%&  \textstyle =- \frac{a Xa}{\sqrt{1-a^2}} Z+  Xa JX + \sqrt{1-a^2} ( - \nabla_Z X - T(X,Z)) + a (J \nabla_X X - \nabla_{JX} X - T(X,JX)) \\
%& \textstyle =- \frac{a Xa}{\sqrt{1-a^2}} Z+  Xa JX + \sqrt{1-a^2} \tau X+ a (J \nabla_X X + Z)  - \nabla_{X_2} X\\
%\end{align*}
%Using that $[X, X_2] + \nabla_{X_2} X = \nabla_X X_2 + T(X_2, X)$

\subsection{Local description} \label{sec:Local}
Working locally, we may assume that $\Sigma = u^{-1}(0)$ is the level set of a function $u: M \to \mathbb{R}$ and that $E$ has a local, positively oriented, orthonormal
basis $Y_1$ and~$Y_2$. Let $\nabla u$ and $\nabla^E u$ be respectively the Riemannian and sub-Riemannian gradient, i.e., $\nabla^E u = Y_1u Y_1 + Y_2 u Y_2$ and $\nabla u = \nabla^E + Zu Z$. We then observe that $T\Sigma$ is spanned by orthonormal basis
\begin{align*}
X & = \frac{1}{\| \nabla^E u\|} (- (Y_2u) Y_1 + (Y_1 u) Y_2), \\
JX & = \frac{1}{\| \nabla^E u\|} (- (Y_1u) Y_1 - (Y_2 u) Y_2), \\
X_2 & = \frac{1}{\| \nabla u\| \|\nabla^E u\|} \left( \| \nabla^E u\|^2 Z - (Zu) (Y_1u) Y_1 - (Zu) (Y_2u) Y_2\right).
\end{align*}
with $X$ being in $E$. We give $\Sigma$ orientation by defining $X$, $X_2$ to be positively oriented. We then see that
$$-d\alpha(X, X_2) = a = \frac{Zu}{\|\nabla u\|}.$$
and hence $\sqrt{1-a^2} = \frac{\| \nabla^E u \|}{\|\nabla u\|}$. Replacing $u$ with $U = \frac{1}{|\nabla u|} u$, we have $U^{-1}(0) = \Sigma$, but now $a = Z U$ and $\sqrt{1-a^2}  = \| \nabla^E U\|$. It follows that
\begin{align*}
K_{\Sigma,E} & =
- \frac{ \langle [X,X_2], X\rangle_h- \langle [X,Z], X \rangle  \| \nabla^E U\| X_2 ZU}{ZU} X_2 ZU \\
 & \qquad -  \frac{(XZU)^2+ \| \nabla^E U \| XZU}{\| \nabla^E U \|^2} +(ZU)^2.
\end{align*}

\begin{example}[The Euclidean unit sphere in the Heisenberg group]
We consider the Heisenberg group as $M = \mathbb{R}^3$ where $E$ has a positively oriented orthonormal basis $A = \partial_x - \frac{1}{2} y \partial_z$ and $B = \partial_y + \frac{1}{2} x \partial_z$. The corresponding Reeb vector field is $Z = \partial_z$. For this case, we can verify that $\tau = 0$ globally. Switching to cylindrical coordinates $(r,\theta,z)$, we see that $E$ is spanned by a positively oriented orthonormal basis
$$\textstyle R = \partial_r , \qquad \Theta = \frac{1}{r} \partial_\theta + \frac{1}{2} r \partial_z.$$
The corresponding contact form is $\alpha = dz - \frac{1}{2} r^2 d\theta$. We have $JR = \Theta$ and $J\Theta = -R$.

Inside the Heisenberg group, we consider $\Sigma$ as the Euclidean sphere which is a level set of $u = r^2 + z^2 -1$. It follows that
$$\textstyle \nabla^E u  =  2r R + rz \Theta, \qquad \nabla u = \nabla^E u + 2z Z, \qquad a= \frac{2z}{\sqrt{r^2(4+z^2) + 4z^2}}$$
and
\begin{align*}
X & \textstyle = \frac{1}{\sqrt{(4+z^2)}} (- z R + 2 \Theta), \\
X_2 & \textstyle = \frac{1}{\sqrt{r(4+z^2) +4 z^2} \sqrt{4+z^2}} \left( r (4+z^2) Z - 4z R - 2  z^2 \Theta \right).
\end{align*}
We see that $\cha(\Sigma)$ are the points where $r=0$. Outside this set, we can use $(z,\theta)$ as global coordinates and have
%$$\partial_r = \frac{dz}{dr} \partial_z = -\frac{r}{\sqrt{1-r^2}} \partial_z = -\frac{\sqrt{1-z^2}}{z} \partial_z.$$
\begin{align*}
a & \textstyle = \frac{2z}{\sqrt{(1-z^2)(4+z^2) + 4z^2}}, \\
\sqrt{1-a^2} & \textstyle = \frac{\sqrt{1-z^2} \sqrt{4+z^2}}{\sqrt{(1-z^2)(4+z^2) + 4z^2}}, \\
X & \textstyle = \frac{2}{\sqrt{(4+z^2)} \sqrt{1-z^2}} ((1-z^2) \partial_z + \partial_\theta  ), \\
X_2 & \textstyle = \frac{1}{\sqrt{(1-z^2)(4+z^2) +4 z^2} \sqrt{4+z^2}} \left( 8 \sqrt{1-z^2} \partial_z - \frac{2z^2}{\sqrt{1-z^2}} \partial_\theta  \right). \\
\end{align*}
% $$X_1 \wedge X_2 = \frac{4}{\sqrt{(1-z^2)(4+z^2) +4 z^2} } \partial_\theta \wedge \partial_z$$
and
\begin{align*}
\sigma & \textstyle = \frac{1}{4} \sqrt{(1-z^2)(4+z^2) +4 z^2} d\theta \wedge dz = \frac{(4 + z^2 - z^4)^2}{8(4 -z^2 +3z^4)}  d\theta \wedge da \\
& \textstyle = \left(\frac{1}{2} + O(|a|-1) \right) d\theta \wedge da 
\end{align*}
as $a \to \pm 1$. It follows that $\int_{1-c \leq |a| \leq 1} \sigma = \pi c + O(c^2)$, and particular, \eqref{A} is satisfied.
We finally have since $\tau = 0$, we have $\langle [X,Z], X \rangle_g =0$ and,
$$\textstyle K_{\Sigma,E} =- \frac{1}{a} \langle [X,X_2], X\rangle X_2a-  \frac{(Xa)^2}{1-a^2} - \frac{Xa}{\sqrt{1-a^2}} +a^2,$$

\end{example}

\subsection{A sufficient condition for \eqref{A}} \label{sec:Asufficient}
We present the following sufficient condition for our Assumption \eqref{A}, in terms of the geometry of the level sets of $a$.
\begin{proposition}
Assume that the following two assumptions hold.
\begin{enumerate}[\rm (i)]
\item 1 is an isolated critical value of $|a|$, i.e., for some $0< c_1 <1$, $\nabla^h a|_x \neq 0$ whenever $c_1 < |a(x)| <1$.
\item For $\Sigma_c$ as in \eqref{Sigmac} with length $\ell(\Sigma_c)$, we have for some constants $C_{\pm}\geq 0$,
$$\ell(\Sigma_c) = C_{\pm}+  O(|c|^{1/2}), \quad \text{as}\quad  c\to \pm 1.$$
\end{enumerate}
Then \eqref{A} holds.
\end{proposition}

\begin{proof}
Define $\tilde \Phi(c) = \int_{\Sigma_{0 \leq c}} a \sigma$ and $\tilde \Phi(-c) = - \int_{\Sigma_{-c\leq 0}} a \sigma$ for $c \geq 0$.
For proving assumption \eqref{A}, it is sufficient to show that it holds for $\tilde \Phi$. By (i),
$$a:\Sigma \to [-1,1]$$
is a submersion on $a^{-1} (c_1, 1)$ which is proper by the compactness of $\Sigma$. It is a fiber bundle by the Ehresmann theorem, hence, we can use $a$ as a coordinate, and consider the fibers to be diffeomorphic to a one-dimensional manifold $F$. Furthermore, on $a^{-1}(c_1, 1)$, write
$$a \sigma = \nu \wedge da, \qquad \text{with} \qquad \nu = \frac{a}{\| \nabla^h a\|^2} ((X_2a) X^*  - (Xa) X_2^*).$$
We then have that $\Phi(1) -\Phi(c_1) = \int_{c_1}^1 \int_{y\in F} \nu(a,y) da$, and hence $\Phi'(c) = \int_{y\in F} \nu(c,y)$.

Next, we need to show that $\Phi'(c)$ is uniformly bounded on $(c_1, 1)$. Since $\Phi'(c) >0$, it is sufficient to show that $\limsup_{c\to 1} \frac{\Phi(1) - \Phi(c)}{1-c} <\infty$. We observe that since $|\beta| = \sqrt{1-a^2}$,
$$\Phi(1)- \Phi(c) = \int_{\Sigma_{c \leq 1}} a \sigma = - \int_{\Sigma_{c\leq 1}} d\beta = - \int_{\Sigma_c} \beta \leq \sqrt{1-c^2} \ell(\Sigma_c),$$
it follows by (ii)
$$\limsup_{c\to 0} \frac{\Phi(1)- \Phi(c)}{1-c} \leq  \limsup_{c\to 0} \ell(\Sigma_c) \frac{\sqrt{1+c}}{\sqrt{1-c}}< \infty.$$
\end{proof}
%We include this following remark.
%\begin{remark}
%We note that for (ii) to hold, we must have that $\lim_{c\to 0} \ell(\Sigma_c) =0$. This is not something we can always assume, since we only know that $\mathrm{char}(\Sigma)$ is contained in a $C^1$-submanifold. However, we mention that \cite[Prop 4.6.11]{Geiges08} guarantees the existence of a $C^\infty$-close perturbation $\Sigma'$ of our surface $\Sigma$ such that the characteristic foliation is Morse-Smale type. On $\Sigma'$, we would only have isolated points as our characteristic set.
%\end{remark}

\section{Geometric identities and the variational metric}  \label{sec:Variational}
\subsection{Variational metrics and the Levi-Civita connection}
We will now define a family of Riemannian metrics $g_\ve$ on $M$ such that $g_\ve |_E = g_E$ and such that $Z$ is orthogonal to $E$ with
$$\textstyle g_\ve(Z, Z) = \langle Z, Z \rangle_\ve = \frac{1}{\ve}.$$
Remark that $\ve = 1$ corresponds to our case in Section~\ref{sec:3dContact}. The tensors $J$ and $\tau$ are still skew-symmetric and symmetric respectively
relative to these metrics. Furthermore, if $\nabla$ is as in Section~\ref{sec:ContactReeb}, then this connection is compatible with $g_\ve$ for every $\ve >0$. 
Our goal first will be to relate the geometry of $(M,g_\ve)$ using the tensors~$J$ and~$\tau$.
For each $g_\ve$, we have a corresponding Levi-Civita connection $\nabla^\ve$.
\begin{lemma} \label{lemma:LeviCivita}
Introduce an operator
$$\textstyle Q_\ve = \frac{1}{2} J - \ve \tau.$$
Then for arbitrary vector fields $W_1$ and $W_2$ in $\Gamma(TM)$,
\begin{equation} \label{nablave} \nabla_{W_1}^\ve W_2 = \nabla_{W_1} W_2 +\left\langle Q_\ve W_1, W_2 \right\rangle Z - \frac{1}{\ve} \alpha(W_2) Q_\ve W_1 - \frac{1}{2\ve} \alpha(W_1) J W_2. \end{equation}
\end{lemma}

\begin{proof}
This follows from simple application of the Koszul formula. If $Y_0, Y_1, Y_2 \in \Gamma(E)$, then
\begin{align*}
\langle \nabla_{Y_0}^\ve Y_1, Y_2 \rangle_\ve & \textstyle = \langle \nabla_{Y_0} Y_1, Y_2 \rangle_\ve,  & \langle \nabla_{Y_0}^\ve Z, Z \rangle_\ve & = 0, \\
\langle \nabla_{Y_0}^\ve Y_1, Z \rangle_\ve & = \textstyle\frac{1}{2\ve} \langle [Y_0,Y_1], Z \rangle - \langle \tau Y_0, Y_1 \rangle,  & \langle \nabla_Z^\ve Y_1, Z \rangle_\ve & =  - \langle \nabla_Z^\ve Z, Y_1 \rangle_\ve = 0 \\
\langle \nabla_{Y_0}^\ve Z, Y_1 \rangle_\ve & = \textstyle\langle \tau  Y_0, Y_1 \rangle - \frac{1}{2\ve} \langle [Y_0, Y_1], Z \rangle, & \langle \nabla_Z^\ve Z, Z \rangle_\ve &=  0.
\end{align*}
and $\langle \nabla_Z^\ve Y_1, Y_2 \rangle_\ve = \langle [Z, Y_1], Y_2 \rangle + \langle \tau Y_1, Y_2 \rangle  - \frac{1}{2\ve}\langle [Y_1, Y_2], Z \rangle$.
The result follows.
\end{proof}

\subsection{Expansion of the curvature}
We want to see how the curvature tensor of~$\nabla^\ve$ changes with respect to $\ve$, described by the relations below.
\begin{lemma}
If $Y$ be a vector field with values in $E$. We then have identities
\begin{align*}
\langle R^\ve(Y,JY) JY, Y \rangle_\ve &= \langle R(Y,JY) JY, Y \rangle - \frac{3}{4\ve} \| Y\|^4  + \frac{\ve}{2} \|\tau\|^2\|Y||^4, \\
\langle R^\ve(Y, JY) Y , Z \rangle_\ve & = \langle (\nabla_{JY} \tau) Y, Y \rangle - \langle (\nabla_{Y} \tau) JY, Y \rangle, \\
\langle R^\ve(Y, Z) Z, Y \rangle_\ve & =  \frac{1}{4\ve^2} \| Y\|^2  -  \| \tau Y \|^2  - \langle (\nabla_Z \tau) Y, Y \rangle - \frac{1}{\ve} \langle \tau Y, JY \rangle, \\
\langle R^\ve(Y, Z) JY, Z \rangle_\ve & = - \frac{1}{\ve} \langle \tau Y , Y \rangle +   \langle (\nabla_Z \tau) Y + \tau^2 Y, JY \rangle .
\end{align*}
\end{lemma}

\begin{proof}
For general vector fields $V_j$, $j=1,2,3,4$, write $\eta_{V_1}^\ve V_2 := \nabla^\ve_{V_1} V_2 - \nabla_{V_1} V_2$, then
$$R^\ve(V_1, V_2)  - R(V_1, V_2)  = (\nabla_{V_1} \eta^\ve)_{V_2} - (\nabla_{V_2} \eta^\ve)_{V_1} + \eta_{T(V_1, V_2)}^\ve + [\eta_{V_1}^\ve,\eta_{V_2}^\ve]. $$
From Lemma~\ref{lemma:LeviCivita}, we have $\eta_{V_1}^\ve V_2 = \left\langle Q_\ve V_1, V_2 \right\rangle Z - \frac{1}{\ve} \alpha(V_2) Q_\ve V_1 - \frac{1}{2\ve} \alpha(V_1) J V_2$, which gives us
\begin{align*}
(\nabla_{V_3}\eta^\ve)_{V_1} V_2 & =  -\ve \left\langle (\nabla_{V_3} \tau) V_1, V_2 \right\rangle Z + \alpha(V_2) (\nabla_{V_3} \tau) V_1, \\
\eta_{T(V_1,V_3)}^\ve V_2 & = \alpha(V_1) \left\langle Q_\ve \tau V_3, V_2 \right\rangle Z -  \frac{1}{\ve} \alpha(V_1) \alpha(V_2) Q_\ve \tau V_3 \\
& \qquad - \alpha(V_3) \left\langle Q_\ve \tau V_1, V_2 \right\rangle Z +  \frac{1}{\ve} \alpha(V_3) \alpha(V_2) Q_\ve \tau V_1 + \frac{1}{2\ve} \langle JV_1, V_3\rangle  J V_2
\end{align*}
Using these identities, we have
\begin{align*}
& \langle R^\ve(Y,JY) JY, Y \rangle_\ve - \langle R(Y,JY) JY, Y \rangle \\
& = - \| Y\|^2 \langle \eta_{Z}^\ve JY, Y \rangle + \langle [\eta_{Y}^\ve , \eta_{JY}^\ve ] JY, Y \rangle \\
& = - \frac{1}{2\ve} \| Y \|^4 - \frac{1}{\ve} \langle Q_\ve JY , JY\rangle \langle Q_\ve Y, Y \rangle + \frac{1}{\ve} \langle Q_\ve Y , JY\rangle \langle Q_\ve JY, Y \rangle \\
& = - \frac{1}{2\ve} \| Y \|^4 + \ve \langle \tau Y , Y\rangle^2  + \ve \langle \tau Y , JY\rangle^2  - \frac{1}{4\ve}\| Y\|^4 \\
& =  - \frac{3}{4\ve} \| Y \|^4 + \ve \| \tau \|^2\| Y\|^4.
\end{align*}
Using that $\nabla$ and hence its curvature operator preserves $E$ and $E^\perp$, then
\begin{align*}
& \langle R^\ve(Y, JY) Y , Z \rangle_\ve = \frac{1}{\ve} \langle (\nabla_{Y} \eta^\ve)_{JY} Y - (\nabla_{JY} \eta^\ve)_{Y} Y - \eta_Z^\ve Y + [\eta^\ve_Y, \eta^\ve_{JY}] Y, Z\rangle \\
& = - \langle  (\nabla_{Y} \tau) JY - (\nabla_{JY} \tau) Y , Y \rangle .
\end{align*}
\begin{align*}
\langle R^\ve(Y, Z) Z, Y \rangle_\ve & =  \langle (\nabla_Y \eta^\ve)_Z Z - (\nabla_Z \eta^\ve)_Y Z - \eta_{\tau Y}^\ve Z + [\eta_Y^\ve, \eta_Z^\ve] Z , Y \rangle \\
& =  \langle  -  (\nabla_Z \tau) Y +  \frac{1}{\ve}  Q_\ve \tau Y  - \frac{1}{2\ve^2} J Q_\ve Y , Y \rangle \\
& = - \langle (\nabla_Z \tau) Y, Y \rangle - \frac{1}{\ve} \langle \tau Y, JY \rangle - \| \tau Y\|^2  + \frac{1}{2\ve^2} \| Y \|^2 .
\end{align*}
\begin{align*}
\langle R^\ve(Y, Z) JY, Z \rangle_\ve & = - \langle  -  (\nabla_Z \tau) Y +  \frac{1}{\ve}  Q_\ve \tau Y  - \frac{1}{2\ve^2} J Q_\ve Y , JY \rangle \\
& =  \langle (\nabla_Z \tau) Y, JY \rangle - \frac{1}{\ve} \langle \tau Y, Y \rangle + \langle \tau Y, \tau JY \rangle.
\end{align*}
The result follows.
\end{proof}

\section{Embedded surfaces and variational metric} \label{sec:Embedded}
In this section, it will be convenient to introduce notation
$$b_\ve = \sqrt{1-(1-\ve) a^2},$$
such that $\frac{1}{b_\ve}$ which is uniformly bounded for $0 < \ve < 1$, while unbounded for $\ve =0$. For our surface $\Sigma \subseteq M$, we write $h_\ve = g_\ve |_{T\Sigma}$ for the induced Riemannian metric from the metric $g_\ve$, with $\sigma^\ve$ being the corresponding volume form. Recall the definition of $\Sigma'$ from \eqref{Sigmac}.
Observe that $T\Sigma'$ will then have a positively oriented, orthonormal basis $X, \hat X_2^\ve$ with respect to $h_\ve$, where
$$\hat X^\ve_2 = \textstyle  \frac{\sqrt{\ve}}{b_\ve} X_2, \quad \text{giving us} \quad \sigma^\ve = \frac{b_\ve}{\sqrt{\ve}}.$$
Using the orientation of $\Sigma$ and $M$ we also have the unit normal vector field,
$$\hat N^\ve = \textstyle \frac{1}{b_\ve} N^\ve \qquad N^\ve = \ve a Z - \sqrt{1-a^2} JX.$$
We denote corresponding scalar-valued second fundamental form by $\II^\ve$.

We want to consider the Gaussian curvature $K^\ve$ of $h_\ve$. It will be sufficient to find the formula on $\Sigma'$. If $R^\ve$ is the curvature operator of~$\nabla^\ve$ then
\begin{align} \label{GaussEq}
K^\ve & = \langle R^\ve(X, X_2^\ve) X_2^\ve, X \rangle_\ve + \II^\ve(X, X) \II^\ve(X_2^\ve, X_2^\ve) -  \II^\ve(X, X_2^\ve)^2, \\ \nonumber
&=: \Sec^\ve + \II_{11}^\ve \II_{22}^\ve - (\II^\ve_{12})^2.
\end{align}
We will use this Gauss Equation to compute the curvature.
\begin{lemma}\label{lemma:2ff}
On $\Sigma'$, we have
$$\textstyle \II^\ve_{11} =  - \frac{1 }{b_\ve} \left( \ve a \tau_0 + b_0 \langle \nabla_X X  ,J X \rangle \right), \qquad \II_{22}^\ve  =  \frac{\ve}{b_\ve} \left(-  \frac{X_2 a}{b_\ve^2 b_0} + a  \tau_0 \right) ,$$
$$ \II^\ve_{12}  = - \frac{\sqrt{\ve} X a}{b_0 b_\ve^2} + \frac{1 - 2\ve \tau_1}{2\sqrt{\ve}} = - \frac{\sqrt{\ve}}{b_\ve^2} \left( b_0 \langle  \nabla_{X_2}X, JX \rangle + a^2 (1+\ve\tau_1) \right)   + \frac{1}{2\sqrt{\ve}} .$$
\end{lemma}

\begin{proof}
If we write $X =X_1$, then for $i,j=1,2$, 
\begin{align*}
&\II^\ve(X_i, X_j) = - \langle \nabla_{X_i}^\ve \hat N_\ve ,X_j \rangle_\ve= - \frac{1}{b_\ve} \langle \nabla_{X_i}^\ve N_\ve ,X_j \rangle_\ve, \\
& \textstyle = - \frac{X_i a}{b_\ve} \left\langle \ve Z + \frac{a}{b_0} JX ,X_j \right\rangle_\ve - \frac{\ve a}{b_\ve} \langle \nabla^\ve_{X_i} Z  ,X_j \rangle_\ve + \frac{b_0}{b_\ve} \langle \nabla_{X_i}^\ve JX  ,X_j \rangle_\ve.
\end{align*}
$$  \frac{\ve a^2}{2\ve b_\ve} \frac{\sqrt{\ve}}{b_\ve} (1- 2\ve \tau_1) + \frac{\sqrt{\ve}}{b_\ve} \frac{b_0^2}{\ve 2b_\ve} (1-2\ve \tau_1)$$

We see that
\begin{align*}
\nabla_X^\ve Z & = \textstyle \frac{1}{2\ve} \left( -JX + 2 \ve \tau_0 X + 2 \ve \tau_1 JX \right), \\
\nabla_{X_2}^\ve Z & =  \textstyle\frac{a}{2\ve}  \left( X - 2 \ve \tau_0 JX + 2\ve \tau_1 X \right), \\
\nabla_{X}^\ve JX & =  \textstyle J \nabla_X X + \frac{1}{2} \left(1 - 2 \ve \tau_1 \right) Z, \\
\nabla_{X_2}^\ve JX & =  \textstyle J \nabla_{X_2}X + a \ve \tau_0 Z  + \frac{1}{2\ve} b_0 X.
\end{align*}
Combining formulas, we get the result.
\end{proof}
An important observation from this lemma is that the expressions $\II^\ve_{11}$, $\II^\ve_{22}$ and $\II^\ve_{12}$ are all uniformly bounded on $\Sigma'$ for fixed $\ve >0$. The same hold for terms $\tau_0$ and $\tau_1$. This observation is a consequence of the fact that $\tau$ and $\II^\ve $ are tensors on the compact $\Sigma$ which have to be uniformly bounded relative to $g_\ve$. From this observation, we obtain the next important corollary.
\begin{corollary} \label{cor:Bounded}
The following expressions are uniformly bounded on $\Sigma'$,
$$\frac{Xa}{b_0} , \qquad \frac{X_2a}{b_0}, \qquad b_0 \langle \nabla_X X, JX \rangle, \qquad b_0 \langle \nabla_{X_2} X, JX \rangle.$$
\end{corollary}

\begin{remark}
Observe that \eqref{XaIdentity} can be reproduced from the formula of $\II_{12}^\ve$. We note that \eqref{XaIdentity} can we rewritten as
\begin{align*}
\frac{Xa}{\sqrt{1 -a^2}}  & = a^2  + \frac{b_0}{a} \langle \nabla_{X_2}^h X ,  X_2 \rangle \\
& \qquad + \frac{b_0}{a} \left\langle  - \langle \frac{1}{2} JX_2 -  \tau X_2 , X\rangle Z + \frac{1}{2} b_0 JX , X_2 \right\rangle -b_0^2 \tau_1. \\
& = a^2  + \frac{b_0}{a} \dv(X) + \frac{b_0}{a} \left( a b_0 (\frac{1}{2} + \tau_1)  + \frac{1}{2} b_0 a  \right) -b_0 \tau_1.\\
& = 1  + \frac{b_0}{a} \dv(X),
\end{align*}
reproducing the observation for the divergence of $X$ in Remark~\ref{re:Characteristic}. 
\end{remark}

\section{Proof of the Gauss-Bonnet formula} \label{sec:ProofGB}
\subsection{The primitive of the curvature integral}
Let $X$ be the orthonormal basis of $E \cap T\Sigma'$ as defined in Section~\ref{sec:Xvector}, with $X_2 = b_0 Z + a JX$. Recall definitions of $\tau_0$ and $\tau_1$ 
We isolate the order for the different terms that will appear in our integral $\int_{\Sigma} K^\ve \sigma^\ve$ for the Gauss-Bonnet theorem. 
\begin{lemma} We have expansion in terms of $\ve$ as
\begin{equation} \label{KB}
  K^\ve \sigma^\ve = \textstyle \frac{\sqrt{\ve}}{b_\ve} \left( \frac{1}{\ve} B_{1,-1} + B_{1,0}  \right) + \left( \frac{\sqrt{\ve}}{b_\ve} \right)^3  \left( \frac{1}{\ve} B_{2,-1}+ B_{2,0}  \right) \\
\end{equation}
with uniformly bounded terms
$$\textstyle B_{1,-1}= \frac{Xa}{b_0} -a^2 , \qquad B_{2,-1} = \langle \nabla_X X  ,J X \rangle X_2 a -  \frac{(Xa)^2}{b_0^2}, \qquad B_{2,0} =a \tau_0 \frac{X_2 a}{b_0}  $$
\begin{align*}
B_{1,0} & = \textstyle \Sec^1 + a^2 \left(    1 -  \tau_0^2     \right)  -  \left( \tau_1 - \frac{1}{2} \right)^2- b_0 a \tau_0 \langle \nabla_X X  ,J X \rangle -2 \tau_1 \frac{Xa}{b_0} ,
\end{align*}
\end{lemma}
\begin{proof}
We consider the Gaussian curvature using the Gauss equation \eqref{GaussEq}. We first have the expansion
\begin{align*}
& \frac{b_\ve^2}{\ve} \Sec^\ve = \langle R^\ve(X, X_2) X_2, X \rangle_\ve \\
&=  a^2 \langle R^\ve(X, JX) JX, X \rangle_\ve + b_0^2 \langle R^\ve(X, Z) Z, X \rangle_\ve -  2ab_0 \langle R^\ve(X, JX) X, Z \rangle_\ve \\
& = a^2 \left(  \langle R^1(X, JX) JX, X \rangle - \frac{3}{4\ve} + \ve (\tau_0^2 +\tau_1^2)  \right)  \\
& \qquad + b_0^2 \left(\frac{1}{4\ve^2} - \tau_0^2 - \tau_1^2 - \langle (\nabla_Z \tau) X, X \rangle - \frac{1}{\ve} \tau_1  \right) \\
& \qquad - 2a b_0 \left( \langle (\nabla_{JX}\tau) X, X \rangle - \langle (\nabla_X \tau ) JX, X \rangle \right) \\
& = \Sec^1 + a^2 \left(    \frac{3}{4} -  (\tau_0^2 +\tau_1^2)  \right)  - \frac{ b_0^2}{4} + b_0^2 \tau_1   + a^2 \left(   - \frac{3}{4\ve} + \ve (\tau_0^2 +\tau_1^2)  \right)  + \frac{b_0^2}{4\ve^2} - \frac{b_0^2}{\ve} \tau_1 \\
& = \Sec^1 + a^2 \left(  - \frac{1}{\ve} +  \frac{3}{4} -  (\tau_0^2 +\tau_1^2) +\tau_1  + \ve \left( \frac{1}{4} + \tau_0^2 +\tau_1^2 -\tau_1 \right)    \right)  \\
& \qquad + b_\ve^2 \left( \frac{1}{4\ve^2}  - \frac{1}{\ve} \tau_1 - \frac{1}{4}  + \tau_1   \right) .  \end{align*}
Writing out the identity $K^\ve \sigma^\ve = \frac{b_\ve}{\sqrt{\ve}} (\Sec^\ve + \II^\ve_{11} \II_{22}^\ve - (\II^\ve_{12})^2) \sigma$, and using Lemma~\ref{lemma:2ff},
\begin{align*}
\frac{b_\ve}{\sqrt{\ve}} \Sec^\ve 
 & = \frac{\sqrt{\ve}}{b_\ve} \left( \Sec^1 + a^2 \left(  - \frac{1}{\ve} +  \frac{3}{4} -  (\tau_0^2 +\tau_1^2) +\tau_1  + \ve \left( \frac{1}{4} + \tau_0^2 +\tau_1^2 -\tau_1 \right)    \right) \right)  \\
& \qquad + \frac{b_\ve}{\sqrt{\ve}} \left( \frac{1}{4\ve}  -  \tau_1 - \ve \left( \frac{1}{4}  - \tau_1 \right)   \right) , \\
\frac{b_\ve}{\sqrt{\ve}} \II_{11}^\ve \II_{22}^\ve & = - \frac{\sqrt{\ve}}{b_\ve}  \left( \ve a^2 \tau_0^2 + b_0 a \tau_0 \langle \nabla_X X  ,J X \rangle \right)  + \frac{\sqrt{\ve}}{b_\ve^3}  \left( \ve a \tau_0 \frac{X_2a}{b_0}  +  \langle \nabla_X X  ,J X \rangle X_2 a \right),
\end{align*}
\begin{align*}
-\frac{b_\ve}{\sqrt{\ve}} (\II_{12}^\ve)^2 & = - \frac{b_\ve}{\sqrt{\ve}}\left(\frac{\ve (Xa)^2}{b_0^2 b_\ve^4} + \frac{1-4 \ve \tau_1 + 4 \ve^2 \tau_1^2}{4\ve} - \frac{Xa(1-2\ve \tau_1)}{b_0 b_\ve^2} \right) \\
& = - \frac{\sqrt{\ve}}{b_\ve^3} \frac{(Xa)^2}{b_0^2} + \frac{\sqrt{\ve} }{b_\ve}\frac{Xa(1-2\ve \tau_1)}{\ve b_0}  - \frac{b_\ve}{\sqrt{\ve}} \frac{1-4 \ve \tau_1 + 4 \ve^2 \tau_1^2}{4\ve} .
\end{align*}
Adding all of these terms together, we have
\begin{align*}
& \frac{\sqrt{\ve}}{b_\ve} \left( \Sec^1 + a^2 \left( 1 -  \tau_0^2     \right) - a^2 \left( \tau_1- \frac{1}{2} \right)^2 - b_0 a \tau_0 \langle \nabla_X X  ,J X \rangle -2 \tau_1 \frac{Xa}{b_0}  \right) \\
 &  + \frac{\sqrt{\ve}}{b_\ve} \left(\frac{Xa}{\ve b_0}   - \frac{a^2}{\ve}   + a^2\ve \left( \tau_1 - \frac{1}{2} \right)^2    \right)   - b_\ve \sqrt{\ve}  \left( \tau_1 -  \frac{1}{2} \right)^2 \\
&  + \frac{\sqrt{\ve}}{b_\ve^3}  \left( \ve a \tau_0 \frac{X_2a}{b_0}  +  \langle \nabla_X X  ,J X \rangle X_2 a -  \frac{(Xa)^2}{b_0^2}  \right) 
\end{align*}
which equals our result. The terms are uniformly bounded by Corollary~\ref{cor:Bounded}.
\end{proof}

\subsection{Proof of Theorem~\ref{th:main}} \label{sec:ProofMain} Let us first state the following result related to limits of integrals that are singular at the limit.
\begin{lemma} \label{lemma:Integrals}
Let $f:[-1,1]\to \mathbb{R}$ be a continuous, uniformly bounded function and write $b_\ve = \sqrt{1-(1-\ve) a^2}$.
\begin{enumerate}[\rm (a)] 
\item The following limits of integrals vanish;
$$\textstyle \lim_{\ve\to 0} \sqrt{\ve} \int_{-1}^1 \frac{f}{b_\ve} da = 0 \quad\text{ and } \quad \lim_{\ve \to 0} \ve \int_{-1}^1\frac{f}{b_\ve^3} da =0.$$
\item We have $\lim_{\ve\to 0} \int_{-1}^1 \frac{f}{b_\ve} da = \int_{-1}^1 \frac{f}{b_0} d\sigma $, with the latter integral is well-defined and finite.
\item The limit $\lim_{\ve \to 0} \sqrt{\ve} \int_{-1}^1 \frac{f}{b_\ve^3} d\sigma$ exits and equals
$$\lim_{\ve \to 0} \sqrt{\ve} \int_{-1}^1 \frac{f}{b_\ve^3} d\sigma = f(-1) + f(1). $$
\item The limit $\lim_{\ve \to 0}\frac{1}{\sqrt{\ve}}\int_{-1}^1 f(a)(\frac{1}{b_\ve} - \frac{1}{b_0} ) da$ exists and equals
$$\lim_{\ve \to 0}\frac{1}{\sqrt{\ve}}\int_{-1}^1 f(a)(\frac{1}{b_\ve} - \frac{1}{b_0} ) da =- f(1) - f(-1).$$
\end{enumerate}
\end{lemma}
\begin{proof}
The statements for the integral $\int_{-1}^1 \frac{f}{b_\ve} da$ follow from the identity
$$\int_{-1}^1 \frac{f(a)}{\sqrt{1 + (\ve -1) a^2}} da = \sqrt{1-\ve}\int_{-\frac{\sin^{-1}\sqrt{1-\ve}}{\sqrt{1-\ve}}}^{\frac{\sin^{-1}\sqrt{1-\ve}}{\sqrt{1-\ve}}} f( \tfrac{1}{\sqrt{1-\ve}} \sin(u)) du,$$
which converge to $\int_{-\frac{\pi}{2}}^{\frac{\pi}{2}} f( \sin(u)) du$.
We similarly have
$$\int_{-1}^1\frac{f(a)}{(1+ (\ve -1) a^2)^{3/2}} da = \int_{- \ve^{-1/2}}^{\ve^{1/2}} f\left( \tfrac{u}{\sqrt{(1-\ve) u^2+1}} \right) du.$$
Obviously, this integral has upper bound for any $0 < k < \frac{1}{\sqrt{\ve}}$,
\begin{align*} & \int_{-k}^k  f\left( \tfrac{u}{\sqrt{1-(1-\ve) u^2}} \right) du + \left(\frac{1}{\sqrt{\ve}} - k \right) \left(\max_{-1\leq a \leq -\frac{k}{\sqrt{(1-\ve)k^2 +1}}} f(a) + \max_{\frac{k}{\sqrt{(1-\ve)k^2 +1}} \leq a \leq 1} f(a) \right) \\
& \leq \int_{-k}^k  f\left( \tfrac{u}{\sqrt{1-(1-\ve) u^2}} \right) du + \left(\frac{1}{\sqrt{\ve}} - k \right) \left(\max_{-1\leq a \leq -\frac{k}{\sqrt{k^2 +1}}} f(a) + \max_{\frac{k}{\sqrt{k^2 +1}} \leq a \leq 1} f(a) \right)   \end{align*}
and an analogous lower bounds bound involving the minimum. As the limit of both of these bounds vanish when multiplied with $\ve$ and letting $\ve \to 0$, we have the second part of (a). Multiplying with $\sqrt{\ve}$ and taking the limit, we obtain
\begin{align*} & \min_{-1\leq a \leq -\frac{k}{\sqrt{k^2 +1}}} f(a) + \min_{\frac{k}{\sqrt{k^2 +1}}  \leq a \leq 1} f(a) \\
&   \leq \lim_{\ve \to 0} \sqrt{\ve} \int_{-1}^1 \frac{f}{b_\ve^3} da \leq \max_{-1\leq a \leq -\frac{k}{\sqrt{k^2 +1}}} f(a) + \max_{\frac{k}{\sqrt{k^2 +1}} \leq a \leq 1} f(a)  \end{align*}
and letting $k \to \infty$, we obtain the result of (c).

Finally, for the result in (d), define for $-1 \leq c_1 < c_2 \leq 1$, the integral $I_\ve(c_1, c_2) = \int_{c_1}^{c_2} \left(\frac{1}{b_\ve} - \frac{1}{b_0} \right) da = F_\ve(c_2) - F_\ve(c_1)$, where $F_\ve(0) = 0$ and $F'_\ve(0) = \frac{1}{b_\ve} - \frac{1}{b_0}$. Then
$$F_\ve(c) = \frac{\sin^{-1}(\sqrt{1-\ve} c)}{\sqrt{1-\ve}} - \sin^{-1}(c),$$
and so
$$\lim_{\ve \to 0} \frac{1}{\sqrt{\ve}} F_\ve(c) =\left.  \frac{\partial}{\partial \sqrt{\ve}} \frac{\sin^{-1}(\sqrt{1-\ve} c)}{\sqrt{1-\ve}} \right|_{\ve =0}
= \left\{
\begin{array}{ll}
-1 & \text{if $c =  1$,} \\
1 & \text{if $c =  -1$,} \\ 0 & \text{otherwise.}
\end{array}
 \right..$$

Hence, for $-1 < c_1 < c_2 < 1$, we have
$$\lim_{\ve \to 0} \frac{I_\ve(1,c_2)}{\sqrt{\ve}} = -1, \qquad \lim_{\ve \to 0} \frac{I_\ve(c_1,c_2)}{\sqrt{\ve}} =0, \qquad \lim_{\ve \to 0} \frac{I_\ve(c_1,1)}{\sqrt{\ve}} =-1.$$ Next, since we have for any subinterval, we have
$$\left( \min_{c_1 \leq a \leq c_2} f(a)\right) I_\ve(c_1, c_2) \leq \int_{1-\rho}^1 f(a)(\frac{1}{b_\ve} - \frac{1}{b_0} ) da \leq \left( \max_{c_1 \leq a \leq c_2} f(a)\right) I_\ve(c_1, c_2),$$
it follows that
$$\lim_{\ve \to 0}\frac{1}{\sqrt{\ve}}\int_{1-\rho}^1 f(a)(\frac{1}{b_\ve} - \frac{1}{b_0} ) da =- f(1) - f(-1).$$
\end{proof}

\begin{proof}[Proof of Theorem~\ref{th:main}]
From the equation \eqref{KB}, define a function $A_{ij}: [-1,1] \to \mathbb{R}$, such that  $A_{i,j}(c) = \int_{0 \leq a \leq c} B_{i,j} \sigma$ for $c \geq 0$, and $A_{i,j}(c) = -\int_{c \leq a \leq 0} B_{i,j} \sigma$ for $c$ negative. Recall that each $B_{ij}$ are uniformly bounded on $\Sigma$ by Corollary~\ref{cor:Bounded}, so from Assumption \eqref{A}, each $A_{i,j}$ is differentiable in when $1 -|c| \in (0,c_1)$ for some $c_1<1$ with a bounded derivative. We look at limits of $\int_{\Sigma} K^\ve \sigma^\ve = 2\pi \chi(\Sigma)$. \\

\paragraph{\it Computations of terms of order $-1/2$} Since $\lim_{\ve \to 0} \sqrt{\ve} \int_{\Sigma} K^\ve \sigma^\ve =0$, then
$$\lim_{\ve \to 0} \sqrt{\ve} \int_{|a| < c_1} K^\ve \sigma^\ve = \int_{|a| < c_1} \frac{1}{b_0} B_{1,-1} \sigma$$
and from Lemma~\ref{lemma:Integrals} and its proof
\begin{align*} & \lim_{\ve \to 0} \sqrt{\ve} \int_{c_1 \leq a \leq 1} K^\ve \sigma^\ve  \\
& = \lim_{\ve \to 0} \int_{c_1 \leq |a| \leq  1} \frac{\ve}{b_\ve} \left( \frac{1}{\ve} A_{1,-1}'  + A_{1,0}'  \right)  da  - \lim_{\ve \to 0} \int_{c_1 \leq |a| \leq 1} \frac{\ve}{b_\ve^3} \left( A_{2,0}'+ A_{2,1}' \ve   \right)  da \\
&  =  \int_{c_1 \leq a \leq 1} \frac{1}{b_0}  A_{1,-1}'    da = \int_{c_1 \leq a \leq 1} \frac{B_{1,-1}}{b_0} \sigma.
\end{align*}
and similarly $ \lim_{\ve \to 0} \sqrt{\ve} \int_{-1 \leq a \leq -c_1} K^\ve \sigma^\ve = \int_{c_1 \leq a \leq 1} \frac{B_{1,-1}}{b_0} \sigma$. In conclusion,
$$\lim_{\ve \to 0} \sqrt{\ve} \int_{\Sigma} K^\ve \sigma^\ve = \int_{\Sigma} \frac{B_{1,-1}}{b_0} \sigma =0.$$

\paragraph{\it Computations of terms of order $0$} Using our computations for order $-1/2$, we have that
\begin{align*} &  \int_\Sigma K^\ve \sigma^\ve = \int_\Sigma K^\ve \sigma^\ve - \int_{\Sigma} \frac{B_{-1,1}}{b_0} \sigma \\
& = \int_{\Sigma}  \left( \left(  \frac{1}{\sqrt{\ve}} B_{1,-1} \left(\frac{1}{b_\ve} - \frac{1}{b_0} \right) +\frac{\sqrt{\ve}}{b_\ve}B_{1,0}   \right) + \frac{\sqrt{\ve}}{b_\ve^3} ( B_{2,-1} +B_{2,0} \ve   ) \right) \sigma \\
& = \int_{-1}^1  \left( \left(  \frac{1}{\sqrt{\ve}} A_{1,-1}' \left(\frac{1}{b_\ve} - \frac{1}{b_0} \right) +\frac{\sqrt{\ve}}{b_\ve}A_{1,0}'   \right) + \frac{\sqrt{\ve}}{b_\ve^3} ( A_{2,-1}' +A_{2,0}' \ve   ) \right) da
\end{align*}
and using Lemma~\ref{lemma:Integrals}, we finally have
$$\lim_{\ve \to 0}  \int_\Sigma K^\ve \sigma^\ve = 2\pi \chi(\Sigma) =A_{2,-1}'(1) - A_{1,-1}'(1) + A_{2,-1}'(-1) - A_{1,-1}'(-1)  .$$

For the final part,
\begin{align*}
B_{1,-1} & =  1- a^2 + \frac{\sqrt{1-a^2}}{a} \dv(X), \\
B_{2,-1} & =  \langle \nabla_X X, X_2 \rangle X_2 a - \frac{(Xa)^2}{b_0^2} = -  \langle  X, \nabla_X X_2 \rangle X_2 a - \frac{(Xa)^2}{b_0^2} \\
& = -  \langle  X, \nabla_X^h X_2 \rangle X_2 a +  \langle  X, b_0 \tau X \rangle X_2 a - \frac{(Xa)^2}{b_0^2} \\
& =  \left(-  \dv(X_2)  +  b_0 \tau_0 \right) X_2 a - \left(  1+ \frac{\sqrt{1-a^2}}{a} \dv(X) \right)^2 \\
\end{align*}
We notice that $K_{\Sigma, E} = B_{2,-1} - B_{1,-1}$ for the result.
\end{proof}

\begin{remark} \label{re:vanish}
We observe as a corollary of the proof of Theorem~\ref{th:main}, since the terms of order $-\frac{1}{2}$ in $\ve$, we will have
$$\int_{\Sigma} \frac{B_{1,-1}}{b_0} \sigma = \int_{\Sigma} \left( \frac{1}{1-a^2 }- \frac{a^2}{\sqrt{1-a^2}} \right) \sigma =\int_{\Sigma} \left(\sqrt{1-a^2} -  \frac{1}{a} \dv(X) \right) \sigma =0.$$
\end{remark}

\section{Analysis of surfaces with boundary} \label{sec:Boundary}
For the final section, we will consider manifolds with boundary. Let $(M, E,g)$ be a sub-Riemannian contact manifold. We will assume that we have a compact $C^2$-surface $\Sigma \subset M$ with a boundary $\partial\Sigma$ that is piecewise $C^2$. Let $h$ be the induced from the metric $g = g_1$. We parametrize $\partial \Sigma$ by a piecewise $C^2$ curve $\gamma:[0,\ell] \to \partial\Sigma$, parametrized by $h$-arc length and positively oriented. Relative to the restriction of the contact form $\beta = \alpha|_{T\Sigma}$, then
$$W = \{ t \in [0,\ell] \, : \, \beta(\dot \gamma(t)) \neq 0\}.$$
\begin{enumerate}[$\bullet$]
\item We say that $t_1 \in W^+$ if $t_1$ is a left limit point of $W$ such that $\beta(\dot \gamma(t_1+)) =0$.
\item We say that $t_2 \in W^-$ if $t_2$ is a right limit point of $W$ such that $\beta(\dot \gamma(t_2-)) =0$.
\end{enumerate}
Note that $W^+$ and $W^-$ are not necessarily disjoint. We will assume that the following holds.
\begin{equation}
\tag{B} \label{Star} \begin{array}{c} \text{If $t_1 \in W^+$ (resp. $t_2 \in W^-$) then} \\
\text{ $\frac{d}{dt} \beta(\dot \gamma)(t_1+) \neq 0$ (resp. $\frac{d}{dt} \beta(\dot \gamma)(t_2 -) \neq 0$).} \end{array}
\end{equation}
The assumption \eqref{Star} has the following geometric interpretation: If the boundary transisions from not being tangent to $E$ to being tangent to $E$, this has to either be an isolated point tangent to $E$ or it must happen at a point where the boundary fails to be $C^2$, see Remark~\ref{re:Star} for more information. Let $X$ and $X_2$ be as in Section~\ref{sec:Xvector}. For any point outside of $\cha(\Sigma)$, define $k_{E}^\ve(y)$ as the $h_\ve$ geodesic curvature of the leaf of $T\Sigma \cap E$ with respect to
$h_\ve$ at $y$, oriented in the direction of~$X$. For $\gamma(t) \not \in \cha(\Sigma)$, write
\begin{equation} \label{angle} \dot \gamma(t) = \cos \theta(t) X + \sin \theta(t) X_2. \end{equation}
Let $S = \{ y_1, \dots, y_N\}$ be the set of points where $\partial \Sigma$ fail to be $C^2$, each with exterior angles $\phi_1^\ve, \dots, \phi_N^\ve$ with respect to $h_\ve$. Write $\phi_j^1 = \phi_j$. We define $S = S_2 \cup S_1 \cup S_0$, where $S_n$ contains the points $y_i = \gamma(c)$ satisfying that precisely $n$ of the vectors $\dot \gamma(c-)$ and $\dot \gamma(c+)$ are in $E_{y_i} \cap T_{y_i} \Sigma$. Define $k_g^\ve$ be the signed geodesic curvature of $\partial \Sigma$ with respect to $h_\ve$ with $k_g^1 = k_g$.

\begin{theorem}[Sub-Riemannian Gauss-Bonnet theorem with boundary] \label{thm:GBBoundary}
For points $y = \gamma(t) \in \partial \Sigma$ where the below functions make sense, define
\begin{align*}
p^{\pm}(y) &= \sign(\langle X, \dot \gamma(t\pm) \rangle), & q^\pm(y) &= \sign(\langle X_2, \dot \gamma(t\pm) \rangle).
\end{align*}
Write the limit if geodesic curvatures as $\lim_{\ve \downarrow}k^\ve_E = k_E^0$ and introduce furthermore
$$\hat W^{\pm} = \{ t \in W^\pm \, :  \gamma(t) \in S, \gamma(t) \not \in \cha(\Sigma) \}.$$
If \eqref{A} and \eqref{Star} hold, and with $\theta(t)$ as in \eqref{angle},
\begin{align*}
2\pi \chi(\Sigma) &=  \left. \frac{d}{dc} \right|_{c=0} \int_{|a| \geq 1-c} K_{\Sigma, E} \sigma  + \int_{\partial \Sigma \cap \cha(\Sigma)} k_g(s) ds + \sum_{y_i \in S_2} \phi_i \\
& \quad  + \frac{\pi}{2} \sum_{y_i \in S_1} \sign(\phi_i)  +
\sum_{y_i \in S_0}
\frac{\pi}{2} (1- q^+(y_i) q^-(y_i)) \sign(\phi_i) \\
& \quad +  \frac{\pi}{2}  \sum_{y \in \gamma(\hat W^+)} p^+(y) q^+(y) - \frac{\pi}{2}  \sum_{y \in \gamma(\hat W^-)} p^-(y) q^-(y)  \\
& \quad +  \frac{\pi}{2}  \sum_{t \in \hat W^+} \frac{ k_E^0(\gamma(t+)) }{\dot \theta(t+)}q^+(\gamma(t)) - \frac{\pi}{2}  \sum_{t \in \hat W^-}  \frac{ k_E^0(\gamma(t-)) }{\dot \theta(t-)} q^-(\gamma(t)).
 \end{align*}
\end{theorem}

\begin{proof}
Our goal will be to show that
$$\textstyle \int_{\Sigma} K^\ve d\sigma^\ve + \int_{\partial \Sigma} k^\ve_g(s^\ve) ds^\ve + \sum_{j=1}^N \phi_j^\ve = C_0 + \frac{C_1}{\sqrt{\ve}} + \frac{C_2 \log \ve}{\sqrt{\ve}} + o(1),$$
for some constants $C_0, C_1, C_2$. It then follows from the Gauss-Bonnet theorem that $C_1 = C_2 =0$ and that $C_0 = 2\pi \chi(\Sigma)$. The conclusion follows from proving that $C_0$ equals the expression in Theorem~\ref{thm:GBBoundary}. The proof will proceed in parts. We first compute geodesic curvature along $C^2$ components of the boundary, then by establishing the limiting behavior of the resulting integrals. Finally, there is an analysis of the limiting behavior of corners.

\subsection*{Geodesic curvature along $C^2$ components of the boundary}
We let $\gamma(t)$ be a parametrization of the boundary $\partial \Sigma$ by $h$-arc length defined on $[0,\ell]$. Write
$$\textstyle \gamma_\ve(s) = \gamma( \varphi_\ve(s) ), \qquad \frac{d}{ds} \varphi_\ve(s) = \frac{1}{\|\gamma(\varphi_\ve(s))\|_\ve}, \qquad \varphi_\ve(0) =0.$$
for its reparametrization by $h_\ve = g_\ve|_{\Sigma}$-arc length defined for $s \in [0, \ell^\ve]$ with $\ell^\ve = \varphi^{-1}_\ve(\ell)$. Let $I^\ve$ denote $\frac{\pi}{2}$ rotation on $T\Sigma$ in the positive direction with respect to $h_\ve$. Observe that for points outside $\cha(\Sigma)$, then for constansts $A_1$, $A_2$, $B_1$, $B_2$,
\begin{align*}
& \textstyle \left\langle A_1 X + A_2 X_2, I^\ve (B_1 X+ B_2 X_2) \right\rangle_\ve  \left\langle A_1 X + A_2 X_2,  B_1 \frac{\sqrt{\ve}}{b_\ve} X_2- B_2 \frac{b_\ve}{\sqrt{\ve}} X \right\rangle_\ve  \\
& \textstyle = \frac{b_\ve}{\sqrt{\ve}} (B_1 A_2 - A_1 B_2) = \frac{b_\ve}{\sqrt{\ve}} \langle A_1 X + A_2 X_2, I^1(B_1 X + B_2 X_2) \rangle.
\end{align*}
For any $s = s^\ve \in [0, \ell^\ve]$, the $h_\ve$-geodesic curvature of $\gamma_\ve$ at $s$ equals $k^\ve_g(s) = \langle D_s^\ve \dot \gamma_{\ve}, I^\ve \dot \gamma_\ve \rangle_\ve$ where $D_t^\ve$ is the covariation derivative with respect to $\nabla^\ve$ along the curve $\gamma(t)$. If $ds^\ve$ denotes the increment with respect to $h_\ve$-arc length, then we are interested in computing the integral
\begin{align*}
\int_{\partial \Sigma} k^\ve_g(s^\ve) ds^\ve & = \int_0^\ell \frac{1}{\|\dot \gamma\|_\ve^2}  \langle  D_t^\ve \dot \gamma , I^\ve \dot \gamma \rangle_\ve \, dt = \int_0^\ell \frac{b_\ve \sqrt{\ve}}{ \ve +(1-\ve)b_0^2\sin^2 \theta }  \langle  D_t^\ve \dot \gamma , I^1 \dot \gamma \rangle \, dt .
\end{align*}
We want to find the part of this integral that has order zero with respect to $\ve$.

\subsection*{Decomposition into subintervals}
Write $[0,\ell] = T = T_0 \cup T_1 \cup T_2 \cup T_3$ where
\begin{align*}
T_0 & = \{ t \in T \, : \, \gamma(t)  \in S\}, \\
T_1 & = \{ t \in T\setminus T_0 \, : \, \gamma(t) \in \cha(\Sigma), \dot \gamma(t) \text{ defined}\} , \\
T_2 & = \{ t \in T\setminus T_0 \, : \, \gamma(t)  \not \in \cha(\Sigma), \dot \gamma(t) \in E_{\gamma(t)} \} = \{t \in T \, : \, \theta = 0\}, \\
T_3 & = \{ t \in T\setminus T_0 \, : \, \gamma(t)  \not \in \cha(\Sigma), \dot \gamma(t) \not \in E_{\gamma(t)} \} = \{t \in T \, : \, \theta \neq 0\}.
\end{align*}
In particular, $T_3$ is an open subset of $T$.  Since $T_0$ consists of isolated points, 
$$\textstyle \int_{\partial \Sigma} k^\ve(s^\ve) ds^\ve = \int_{s^{-1}(T_1 \cup T_2 \cup T_3)} k^\ve(s^\ve) ds^\ve.$$
Furthermore, on $T_1$, we have that $\int_{s^{-1}(T_1)} k^\ve(s^\ve) ds^\ve = \int_{s^{-1}(T_1)} k^1(s) ds$. We will thus only consider $T_2 \cup T_3$. 

Introduce the function $\psi(t) = \beta(\dot \gamma(t)) = b_0(\gamma(t)) \sin \theta(t)$, and note that
$$\textstyle \|\dot \gamma\|_{\ve}^2 = \cos^2 \theta + \frac{b_\ve^2}{\ve} \sin^2 \theta = 1 + (1-\ve) \frac{\psi^2}{\ve}.$$
Since $\gamma$ is piecewise $C^2$, it follows that $\psi$ is $C^1$ on each piecewise component where it is defined.

For $t \in T_2 \cup T_3$, we can then write $\dot \gamma(t)$ as in \eqref{angle}, meaning that $I^1 \dot \gamma(t) = \cos \theta X_2 - \sin \theta X$. Define $D_t$ as the covariant derivative with respect to the tangential connection on $\Sigma$ from $\nabla$. We then obtain
\begin{align*}
D_t^\ve \dot \gamma & =D_t \dot \gamma + \pr_{T\Sigma} \left( \langle Q_\ve \dot \gamma, \dot \gamma \rangle Z - \beta(\dot \gamma) \left(\frac{1}{\ve} Q_\ve  + \frac{1}{2\ve} J \right)\dot \gamma  \right)\\
& = D_t \dot \gamma -\ve  b_0 \langle \tau \dot \gamma, \dot \gamma \rangle X_2 - a b_0\frac{\sin \theta}{\ve} (- \sin \theta X + \cos \theta X_2)  +  b_0 \sin \theta  \tau \dot \gamma \\
& = D_t \dot \gamma +  \psi  \tau \dot \gamma -\ve  b_0 \langle \tau \dot \gamma, \dot \gamma \rangle X_2 - \frac{1}{\ve} a \psi I^1 \dot \gamma,
\end{align*}
leading to
\begin{align*}
\langle  D_t^\ve \dot \gamma(t) , I^1 \dot \gamma(t) \rangle 
& =  \langle D_t \dot \gamma +  \psi  \tau \dot \gamma, I^1 \gamma \rangle  -\ve b_0 \langle \tau \dot \gamma, \dot \gamma \rangle  \cos \theta - \frac{1}{\ve} a \psi \\
& =: k_\Sigma  - \ve b_0 \langle \tau \dot \gamma, \dot \gamma \rangle \cos \theta - \frac{1}{\ve} a\psi .
\end{align*}
We remark that for the special case when $\dot \gamma(t) = X|_{\gamma(t)}$, from the above formula $k^\ve_E$ has a well-defined limit as $\ve \to 0$.

For the integral over $T_2$, we have $\psi = 0$ and our integral becomes
\begin{align*}
\int_{s^{-1}(T_2)} k^\ve ds^\ve & = \int_{T_2} \frac{b_\ve}{\sqrt{\ve}} \left( k_\Sigma - \ve b_0 \langle \tau \dot \gamma, \dot \gamma \rangle \cos \theta \right) \, dt
 = \frac{1}{\sqrt{\ve}} \int_{T_2} b_0  k_\Sigma \, dt + O(\sqrt{\ve}). 
\end{align*}
For the interval $T_3$, we have
\begin{align} \label{T_3}
 \int_{s^{-1}(T_3)} k^\ve ds^\ve  & = \int_{T_3} \frac{\sqrt{\ve} b_\ve}{\ve + (1-\ve) \psi^2} \left( k_\Sigma - \ve b_0 \langle \tau \dot \gamma, \dot \gamma \rangle \cos \theta \right)  \, dt \\ \nonumber
& \qquad  - \frac{1}{\sqrt{\ve}} \int_{T_3} \frac{a b_\ve \psi }{\ve + (1-\ve) \psi^2} \, dt. \end{align}

\subsection*{Computations over $T_3$}
In order to find the integral of the geodesic curvature under $T_3$, we will first need to consider integrals of the form $\int_{ T_3} \frac{\sqrt{\ve} f}{\ve + (1-\ve) \psi^2} \, dt$ and $\frac{1}{\sqrt{\ve}}\int_{ T_3} \frac{f \psi }{\ve + (1-\ve) \psi^2} \, dt$ with $f$ being $C^1$. Consider three types of subsets $L_0$, $L_-$ and $L_+$ of $T_3$, with the following properties.
\begin{enumerate}[$\bullet$]
\item $L_0$ is an interval where $\psi$ is bounded away from zero.
\item $L_+ = (c_1, c_2)$ is an open interval, where $\psi$ is bounded away from zero on any interval $(c_1+\rho, c_2)$, $\rho >0$, but $\psi(c_1+) = 0$. Furthermore, we assume that $\dot \psi$ is bounded away from zero on $L_+$.
\item $L_- = (c_3, c_4)$ is an open interval, where $\psi$ is bounded away from zero on any interval $(c_3, c_4-\rho)$, $\rho >0$, but $\psi(c_4-) = 0$. Furthermore, we assume that $\dot \psi$ is bounded away from zero on $L_-$.
\end{enumerate}
By our assumption \eqref{Star}, we can decompose our set $T_3$ into a disjoint union intervals of the above type, where we have one interval of the type $L_+$ (resp. $L_-$) for every $c \in W^+$ (resp. $ c\in W^-$).

For an interval of the type $L_0$, we have
\begin{align*}
\int_{L_0} \frac{\sqrt{\ve} f}{\ve + (1-\ve) \psi^2} \, dt & = O(\sqrt{\ve}), \\
\frac{1}{\sqrt{\ve}}\int_{L_0} \frac{f \psi }{\ve + (1-\ve) \psi^2} \, dt &= \frac{1}{\sqrt{\ve}}\int_{L_0} \frac{f}{\psi} dt + O(\sqrt{\ve}).
\end{align*}
For $L_+ = (c_1,c_2)$, we will use that for any sufficiently small $\rho >0$, $(c_1,c_2) \setminus (c_1, c_1 + \rho)$ is an interval of the type $L_0$. We hence have that
\begin{align*}
& \int_{c_1}^{c_2} \frac{\sqrt{\ve} f}{\ve+ (1-\ve)\psi^2} \, dt = \int_{c_1}^{c_1+\rho} \frac{\sqrt{\ve} f}{\ve+ (1-\ve)\psi^2} \, dt + O(\sqrt{\ve}),
\end{align*}
and furthermore,
\begin{align*}
&  \frac{\left(\inf_{c_1< t\leq c_1+\rho} \frac{f(t)}{\dot \psi(t)} \right)}{\sqrt{1-\ve}} \left( \tan^{-1} \sqrt{\frac{1-\ve}{\ve}} \psi(c_1+\rho) \right) \\
& \leq \int_{c_1}^{c_1+\rho} \frac{\sqrt{\ve} f}{\ve+ (1-\ve)\psi^2} \, dt \leq \frac{\left(\sup_{c_1 < t\leq c_1+\rho} \frac{f(t)}{\dot \psi(t)} \right)}{\sqrt{1-\ve}} \left( \tan^{-1} \sqrt{\frac{1-\ve}{\ve}} \psi(c_1+\rho) \right).
\end{align*}
Taking a limit as $\ve \to 0$, we have that
\begin{align*}
& (\sign \psi(c_1 +\rho)) \frac{\pi}{2} \left(\inf_{c_1< t\leq c_1+\rho} \frac{f(t)}{\dot \psi(t)} \right)  \\
& \leq \lim_{\ve \downarrow 0} \int_{c_1}^{c_1+\rho} \frac{\sqrt{\ve} f}{\ve+ (1-\ve)\psi^2} \, dt \leq (\sign \psi(c_1 +\rho))\frac{\pi}{2} \left(\sup_{c_1 < t\leq c_1+\rho} \frac{f(t)}{\dot \psi(t)} \right).
\end{align*}
As this should be valid for any $\rho$, we can let $\rho \downarrow 0$ to obtain
\begin{align*}
& \int_{c_1}^{c_2} \frac{\sqrt{\ve} f}{\ve+ (1-\ve)\psi^2} \, dt = (\sign \psi(c_1 +))\frac{\pi}{2} \frac{f(c_1+)}{\dot \psi(c_1+)} + o(1) = \frac{\pi}{2} \frac{f(c_1+)}{\frac{d|\psi|}{dt}(c_1+)} +o(1).
\end{align*}
Similarly, using integration by parts, we find that
\begin{align*}
& \frac{1}{\sqrt{\ve}}\int_{c_1}^{c_2} \frac{f \psi }{\ve + (1-\ve) \psi^2} \, dt = \frac{1}{\sqrt{\ve}}\int_{c_1}^{c_1+\rho} \frac{f \psi }{\ve + (1-\ve) \psi^2} \, dt + \frac{1}{\sqrt{\ve}}\int_{c_1+\rho}^{c_2} \frac{f}{\psi} \, dt + O(\sqrt{\ve}) \\
&= \frac{\log(\ve + (1-\ve) \psi(c_1+\rho)^2)}{2\sqrt{\ve}(1-\ve)} \frac{f(c_1+\rho)}{\dot \psi(c_1+\rho)} - \frac{\log(\ve )}{2\sqrt{\ve}(1-\ve)} \frac{f(c_1+)}{\dot \psi(c_1 +)} \\
& \qquad - \frac{1}{2\sqrt{\ve}(1-\ve)} \int_{c_1}^{c_1+\rho}  \log(\ve + (1-\ve) \psi^2) \frac{d}{dt} \left( \frac{f}{\dot \psi} \right) \, dt \\
& \qquad  + \frac{\log| \psi(c_2)|}{\sqrt{\ve}} \frac{f(c_2)}{\dot \psi(c_2)} - \frac{\log| \psi(c_1+\rho)|}{\sqrt{\ve}} \frac{f(c_1+\rho)}{\dot \psi(c_1+\rho)} \\
& \qquad - \frac{1}{\sqrt{\ve}} \int_{c_1+\rho}^{c_2}  \log|\psi| \frac{d}{dt} \left( \frac{f}{\dot \psi} \right) \, dt + O(\sqrt{\ve}).
\end{align*}
Using that
\begin{itemize}
\item $\frac{\log(\ve + (1-\ve) \psi(c_1+\rho)^2)}{2\sqrt{\ve}(1-\ve)} \frac{f(c_1+\rho)}{\dot \psi(c_1+\rho)} - \frac{\log| \psi(c_1+\rho)|}{\sqrt{\ve}} \frac{f(c_1+\rho)}{\dot \psi(c_1+\rho)}  = O(\sqrt{\ve})$,
\item the integral $ \int_{c_1+\rho}^{c_2}  \log|\psi| \frac{d}{dt} \left( \frac{f}{\dot \psi} \right) \, dt$ is finite,
\item $\lim_{\rho \downarrow 0} \int_{c_1}^{c_1+\rho}  \log(\ve + (1-\ve) \psi^2) \frac{d}{dt} \left( \frac{f}{\dot \psi} \right) \, dt = 0$,
\end{itemize}
we obtain
\begin{align*}
& \frac{1}{\sqrt{\ve}}\int_{c_1}^{c_2} \frac{f \psi }{\ve + (1-\ve) \psi^2} \, dt  \\
&= - \frac{\log(\ve )}{\sqrt{\ve}} \frac{f(c_1+)}{2\dot \psi(c_1 +)}   + \frac{1}{\sqrt{\ve}} \left( \frac{f(c_2)}{\dot \psi(c_2)} \log| \psi(c_2)| -   \int_{c_1}^{c_2}  \log|\psi| \frac{d}{dt} \left( \frac{f}{\dot \psi} \right) \, dt \right) + o(1).
\end{align*}
In particular, there are no terms of order zero for such integrals.

If we do similar computations for $L_- = (c_3,c_4)$, we obtain
\begin{align*}
\int_{c_3}^{c_4} \frac{\sqrt{\ve} f}{\ve+ (1-\ve)\psi^2} \, dt & =  - \frac{\pi}{2} \frac{f(c_4-)}{\frac{d|\psi|}{dt}(c_4-)} +o(1).
\end{align*}
and
\begin{align*}
& \frac{1}{\sqrt{\ve}}\int_{c_3}^{c_4} \frac{f \psi }{\ve + (1-\ve) \psi^2} \, dt  \\
&=  \frac{\log(\ve )}{\sqrt{\ve}} \frac{f(c_4-)}{2\dot \psi(c_4-)}   + \frac{1}{\sqrt{\ve}} \left( -\frac{f(c_3)}{\dot \psi(c_3)} \log| \psi(c_3)| -   \int_{c_3}^{c_4}  \log|\psi| \frac{d}{dt} \left( \frac{f}{\dot \psi} \right) \, dt \right) + o(1).
\end{align*}

Summarizing these computation, we see that the only terms of degree zero in $\ve$ of \eqref{T_3} are,
\begin{align*}
& \frac{\pi}{2}  \sum_{c \in W^+} \frac{b_0(c+) k_{\Sigma}(c+) }{\frac{d|\psi|}{dt}(c+)} - \frac{\pi}{2}  \sum_{c \in W^-} \frac{b_0(c-) k_{\Sigma}(c-) }{\frac{d|\psi|}{dt}(c-)} \\
& = \frac{\pi}{2}  \sum_{c \in W^+ \setminus \gamma(\cha(\Sigma))} \frac{b_0(c+) k_{\Sigma}(c-) }{\frac{d|\psi|}{dt}(c-)} - \frac{\pi}{2}  \sum_{c \in W^- \setminus \gamma(\cha(\Sigma))} \frac{b_0(c-) k_{\Sigma}(c) }{\frac{d|\psi|}{dt}(c+)} \\
& = \frac{\pi}{2}  \sum_{c \in W^+ \setminus \gamma(\cha(\Sigma))} \frac{ k_{\Sigma}(c+) }{\frac{d|\sin \theta|}{dt} (c+)} - \frac{\pi}{2}  \sum_{c \in W^- \setminus \gamma(\cha(\Sigma))}  \frac{ k_{\Sigma}(-c) }{\frac{d|\sin \theta|}{dt} (c-)} \\
& = \frac{\pi}{2}  \sum_{c \in \hat W^+} \frac{ k_{\Sigma}(c+) }{\frac{d|\sin \theta|}{dt} (c+)} - \frac{\pi}{2}  \sum_{c \in \hat W^- }  \frac{ k_{\Sigma}(c-) }{\frac{d|\sin \theta|}{dt} (c-)} .
\end{align*}
In the last equality, we have used that from assumption \eqref{Star}, we have that if $c \in W^{\pm}$, then either $c \in W^+ \cap W^-$
or $c \in W^{\pm} \cap S$. We see that for $c \in W^{\pm}$,
\begin{align*}
\frac{d|\sin \theta|}{dt} (c\pm) & = \sign(\sin \theta(c\pm)) \cos(\theta(c\pm)) \dot \theta(c\pm) \\
& = \dot \theta(c \pm)  q^{\pm}(\gamma(c)) p^{\pm}(\gamma(c)),\\
k_\Sigma(c\pm)  &= \dot \theta(c\pm) + p^{\pm}(\gamma(c\pm)) k_E^0(\gamma(t)).
\end{align*}
In conclusion, we have
\begin{align*}
& \frac{\pi}{2}  \sum_{c \in \hat W^+} \frac{ k_{\Sigma}(c+) }{\frac{d|\sin \theta|}{dt} (c+)} - \frac{\pi}{2}  \sum_{c \in \hat W^- }  \frac{ k_{\Sigma}(c-) }{\frac{d|\sin \theta|}{dt} (c-)} \\
& =   \frac{\pi}{2}  \sum_{y \in \gamma(\hat W^+)} p^+(y) q^+(y) - \frac{\pi}{2}  \sum_{y \in \gamma(\hat W^-)} p^-(y) q^-(y)  \\
& \quad +  \frac{\pi}{2}  \sum_{c \in \hat W^+} \frac{ k_E^0(\gamma(c+)) }{\dot \theta(c+)}q^+(\gamma(c)) - \frac{\pi}{2}  \sum_{c \in \hat W^-}  \frac{ k_E^0(\gamma(c-)) }{\dot \theta(c-)} q^-(\gamma(c)).\end{align*}
This completes the computations on $T_3$.

\subsection*{Contributions from corners}
Finally, we consider elements of $S$. We observe that if $v,w \in T\Sigma$ with oriented angle $\phi_j^\ve$ relative to $h_\ve$, then
$$\lim_{\ve \to 0} |\phi^\ve| = \lim_{\ve \to 0} \cos^{-1} \frac{|\langle v, w \rangle_{h_\ve}|}{\|v\|_{h_\ve} \|w\|_{h_\ve}} = \left\{ \begin{array}{ll}
|\phi_j|, & v,w \in E \cap T\Sigma, \\
\frac{\pi}{2}, & v \in E \cap T\Sigma, w \not \in E \cap T\Sigma, \\
\frac{\pi}{2} (1-s), & v,w \not \in E \cap T\Sigma, s = \sign(\beta(v)\beta(w)).
\end{array} \right.$$

The result again follows by writing the Gauss-Bonnet formula with boundary for the $g_\ve$ metric and taking the limit $\ve \rightarrow 0$.
\end{proof}

\begin{figure}[h!]
\includegraphics[width=1.4in]{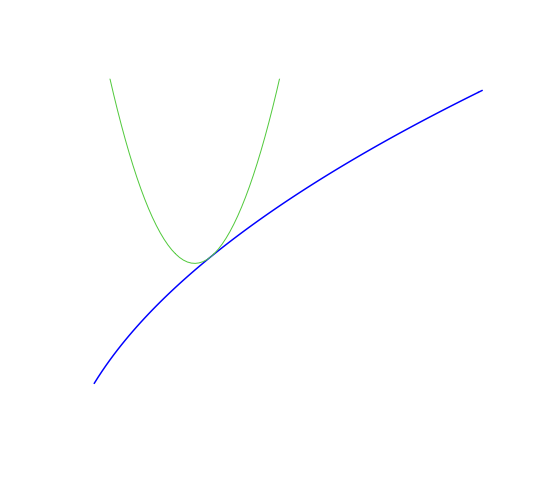}
\includegraphics[width=1.4in]{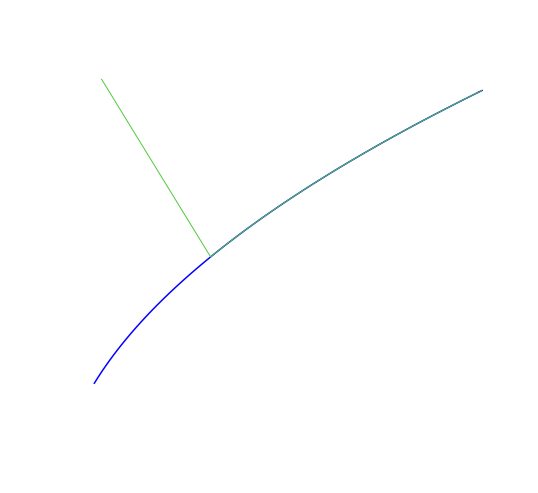}
\includegraphics[width=1.4in]{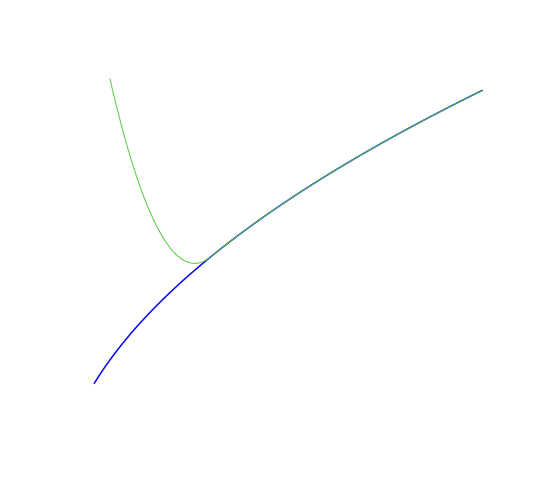}
\includegraphics[width=1.4in]{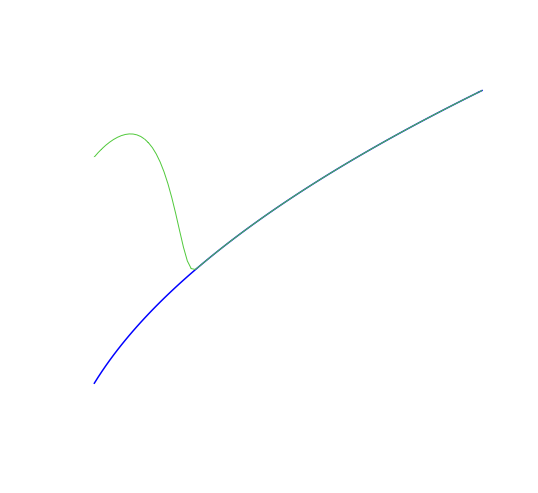}
\caption{The figure shows four cases where the boundary $\partial \Sigma$ in green intersects the characteristic foliation tangent to $E$ in blue. The three first are compatible with \eqref{Star}. In the first it is only tangent at an isolated point, and its second derivative does not follow the curve in blue. The second and the third represent respectively a $C^1$- and a $C^2$-singularity. In the fourth picture, the boundary smoothly becomes tangent to $E$ which is not compatible with \eqref{Star}.}
\end{figure}

\begin{remark} \label{re:Star}
If \eqref{Star} does not hold, these we can easily find examples of $\psi(t)$ such that the integral $\int_{c_1}^{c_1+\rho} \frac{\sqrt{\ve}}{\ve+\psi(t)^2} \, dt$ approach $\infty$ as $\ve \to \infty$, e.g. $\psi(t) = Ct^2$. For finding a Gauss-Bonnet formula in this case, one would need to establish exactly which part of the integral $\int_{c_1}^{c_2} \frac{\sqrt{\ve} \dot \theta b_\ve}{\ve  + (1-\ve) \psi^2}  dt$ has order $0$ relative to $\ve$ for any general~$\psi$. We leave this problem for future research, but conjecture that no such terms exist and therefore Theorem \ref{thm:GBBoundary} holds even when \eqref{Star} does not.
\end{remark}

\bibliography{Bibliography}

\end{document}